\documentclass[12pt]{amsart}

\usepackage{amssymb}
\usepackage{amsmath}
\usepackage{amsthm}
\usepackage{latexsym}
\usepackage{amsfonts}
\usepackage{graphicx}
\usepackage{graphics}
\usepackage{verbatim}
\usepackage[dvipsnames,usenames]{color}

\newcommand{\n}[1]{\|#1\|}
\newtheorem{thm}{Theorem}[section]   
\newtheorem{prop}[thm]{Proposition}
\newtheorem{lem}[thm]{Lemma}

\newtheorem{Def}[thm]{Definition}

\theoremstyle{definition}

\newcommand{\el}{\ell}

\newcommand{\ra}{\;\rightarrow\;}

\newcommand{\al}{\alpha}
\newcommand{\bi}{\beta}

\newcommand{\Ga} {{\varGamma}}

\newcommand{\de}{\delta }
\newcommand{\OO} {{\varOmega}}
\newcommand{\De} {{\varDelta}}
\newcommand{\e}{\varepsilon }
\newcommand{\f}{\varphi}
\newcommand{\Fi}{\varPhi}

\newcommand{\zi}{\zeta }
\newcommand{\thi}{\theta }

\newcommand{\C}{\mathbb{C}}
\newcommand{\T}{\mathbb{T}}
\newcommand{\R}{\mathbb{R}}

\newcommand{\Q}{\mathbb{Q}}

\newcommand{\ld}{\ldots}

\newcommand{\sm}{\smallsetminus}

 \newcommand{\Arg}{\mbox{Arg}}

 \renewcommand{\Im}{\mbox{Im}}

 \newcommand{\dist}{\mbox{dist}}

\begin{document}
%
\title[total unboundedness]{ Generic non-extendability and total unboundedness in  function spaces}
%
%
\author{V. Nestoridis}
\address{National and Kapodistrian University of Athens,
Department of Mathematics,
Panepistemiopolis, 15784,
Athens,
Greece}
\email{vnestor@math.uoa.gr}

\author{A. G. Siskakis}
\address{Department of Mathematics,
University of Thessaloniki, 54124
Thessaloniki, Greece}
\email{siskakis@math.auth.gr}

\author{A. Stavrianidi}
\address{National and Kapodistrian University of Athens,
Department of Mathematics,
Panepistemiopolis, 15784,
Athens,
Greece}
\email{aleksandrastavrianidi@yahoo.gr}

\author{S. Vlachos}
\address{National and Kapodistrian University of Athens,
Department of Mathematics,
Panepistemiopolis, 15784,
Athens,
Greece}
\email{steliosgerrard8@hotmail.gr}

%
\begin{abstract}
For a function space $X(\OO)$ satisfying weak assumptions we prove that the generic function in $X(\OO)$ is totally unbounded, hence non-extendable. We  provide several examples of such spaces; they are mainly localized versions of classical function spaces and intersections  of them.
\end{abstract}



\maketitle

\section{Introduction}\label{sec1}

\noindent

In \cite{5} a function space $X=X(\OO)$ satisfying weak assumptions was considered where $\OO$ is a domain in $\C^d$ and $X(\OO)$ consists  of holomorphic functions on $\OO$. For these spaces $X$ it was proved that the set of non-extendable functions in $X$ is either empty or a $G_\de$ and dense subset of $X$. In the present paper we consider a stronger notion than non-extendability; this is the total unboundedness of a function $f\in X$ or of one of its derivatives. A function $f:\OO\ra\C$ is called totally unbounded if for every ball $B(\zi,\e)$ with radius $\e>0$ centered at $\zi\in\partial\OO$ and every component $Y$ of $\OO\cap B(\zi,\e)$ the restriction $f|_Y$ is unbounded (\cite{2}). If $f$ or one of its derivatives is totally unbounded then $f$ is non-extendable. The converse is not true since there exist non-extendable functions in $A^\infty$, the space of holomorphic  functions on the unit disc $D$ whose derivatives of every order extend continuously on the closed disc, see \cite{6}.

We prove that if $X$ satisfies some weak assumptions, then the set of totally unbounded functions in $X$ is a $G_\de$ and dense subset of $X$ if and only if
 it contains sufficiently many "locally unbounded" functions, i.e. if the following holds:

For every ball $B(\zi,\e)$ with radius $\e>0$ centered at $\zi\in\partial\OO$ and every component $Y$ of $\OO\cap B(\zi,\e)$, there exists a function $h_Y$ in $X$ such that its restriction $h_Y|_Y$ is unbounded.

We also obtain a version of the previous fact for the derivatives of functions in $X$.

The assumption for the space $X=X(\OO)$ is that it is a topological vector space  under the usual operations of addition and scalar multiplication, whose  topology is induced by a complete metric such that  convergence in $X$ implies uniform convergence on compact subsets  of $\OO$. The related description and the main result is in section \ref{sec2}.

In the rest of the article  we give examples of function spaces that satisfy the required conditions.  Most of these spaces are localized versions of classical function spaces. Very often a classical function space consists of all holomorphic in $\OO$ functions satisfying a property $P$ when we approach from $\OO$ the whole boundary $\partial\OO$. For the localization we consider a denumerable set of properties $P_i$, $i\in I$ and we demand that each one of them holds when we approach from $\OO$ only a part $J_i$, $i\in I$ of the boundary $\partial\OO$. These spaces endowed with their natural topology are often Fr\'{e}ch\'{e}t spaces and satisfy our requirements. In order to assure the existence of the functions $h_Y$ in $X$ we impose some geometrical or topological assumptions on $\OO$. One such assumption is that for some $\al\in\partial\OO$ there exists $\bi$ so that the segment $(\al,\bi]$ is disjoint from $\overline{\OO}$ or $\OO$. The essential point is that $\arg\frac{z-\al}{z-\bi}$ has a bounded on $\OO$ branch. Thus, the segment $(\al,\bi]$ can be replaced by other curves not turning around very much.

In the last section we give an example of a space $X$ which is not a localization of a classical function space, but is itself a classical function space. This is the space of holomorphic functions $f(z)=\sum^\infty_{n=0}a_nz^n$ on the open unit disc $D$ such that $\sum^\infty_{n=0}|a_n|^p<\infty$. We prove that for the generic function $f$ in these spaces its derivative $f'$ is totally unbounded and $f$ is non-extendable.

In the future we will consider localizations of the Nevanlinna class or of subclasses of it, as well as, of the Dirichlet space. We mention that the localized Hardy space $H^1$ appears naturally with the results of \cite{3}.
More precisely if $\OO$ is a Jordan domain in $\C$ and $\Fi:D\ra\OO$ is a Riemann map and for a closed arc $J=\{e^{i\thi}:A\le\thi\le B\}\subset \T=\partial D$ with $A<B\le A+2\pi$ the curve $\Fi(J)$ has finite length, then
$
\sup_{0\le r<1}\int^\bi_\al|\Fi'(re^{i\thi})|d\thi<\infty
$
for all $\al,\bi$ such that $A<\al<\bi<B$; that is, the derivative $\Fi'$ belongs to the localized Hardy space $H^1(D,\{e^{i\thi}:\al\le\thi\le\bi\})$.

In order to obtain an ``if and only if'' statement is suffices to consider the localized Hardy space $H^1(D,O)$, where $O\subset \T$ is relatively open. Then $\Fi'\in H^1(D,O)$ if and only if the curve $\Fi(\{e^{i\thi}:a\le\thi\le\bi\})$ has finite length for all $\al,\bi$ such that $\{e^{i\thi}:a\le\thi\le\bi\}\subset O$.
For the definitions of the localized Hardy spaces see section \ref{sec6} below.

Finally it is interesting to investigate properties of the new localized spaces or of the functions belonging to them. What about their zeros or their growth; are there non-tangential limits? Are also polynomials or rational functions dense in these spaces? We also mention that some of the facts established in this paper hold if $X=X(\OO)$ is not a space of holomorphic functions but also a space of harmonic functions on a domain $\OO$ in $\R^d$.
\section{Extendability and total unboundedness}\label{sec2}
\noindent

In this section we present some facts about non extendable and totally unbounded functions that will be central in the paper.
\begin{Def}\label{Def2.1}
Let $\OO\subset\C^d$ be a domain and $f$ a holomorphic function in $\OO$. The function $f$ is called extendable if there exist a domain $U\subset\C^d$, $\partial\OO\cap U\neq\emptyset$ and a holomorphic function $F$ in $U$ such that for some connected component $Y$ of $U\cap\OO$ we have $f|_Y=F|_Y$. Otherwise $f$ is called non-extendable.
\end{Def}

We will also need the following fact

\begin{lem}\label{lem2.2}
(\cite{2}, \cite{5}). Let $\OO\subset\C^d$ be a domain and $U\subset\C^d$ a domain such that $U\cap\partial\OO\neq\emptyset$. Let $Y$ denote a connected component of $U\cap\OO$. Then $\overline{Y}\cap U\cap\partial\OO\neq\emptyset$.
\end{lem}
\begin{prop}\label{prop2.3}
(\cite{5}). Let $\OO\subset\C^d$ be a domain and $f:\OO\ra\C$ a holomorphic function. Then the following are equivalent.

(i) $f$ is extendable

(ii) There exist two balls $b,B$ in $\C^d$ such that $\overline{b}\subset B\cap\OO$, $B\cap\OO^c\neq\emptyset$ and a bounded holomorphic function $F:B\ra\C$ such that $F|_b=f|_b$.
\end{prop}
\begin{prop}\label{prop2.4}
Let $\OO\subset\C^d$ be a domain and $f:\OO\ra\C$ a holomorphic function. Let $T$ denote the identity or an operator of mixed partial derivation.
If $Tf$ is non-extendable, then $f$ is non-extendable as well.
\end{prop}
\begin{proof}
If $f$ is extendable, then according to Definition \ref{Def2.1} there exists a domain $U\subset\C^d$, $\partial\OO\cap U=\emptyset$ and a function $F$ holomorphic in $U$ such that $F|_Y=f|_Y$ for some component $Y$ of $U\cap\OO$. Then obviously $TF|_Y=Tf|_Y$ since $Y$ is open. It follows that   $Tf$  is extendable, which contradicts our hypothesis.
\end{proof}
\begin{thm}\label{thm2.5} (\cite{5}). Let $\OO\subset\C^d$ be a domain and $X=X(\OO)$ be a topological vector space consisting of holomorphic functions on $\OO$, whose topology is induced by a complete metric and such that convergence in $X$ implies uniform convergence on compacta of $\OO$. Then the set of non-extendable functions in $X$ is either $\emptyset$ or a dense and $G_\de$ subset of $X$.
\end{thm}

A slightly stronger notion is that of totally unbounded function (\cite{2}).
\begin{Def}\label{Def2.6}
Let $\OO\subset\C^d$ be a domain. A function  $g:\OO\ra\C$  is called totally unbounded on $\OO$, if for every ball $B(\zi,\e)\subset\C^d$ with radius $\e>0$ and centered at $\zi\in\partial\OO$ and every component $Y$ of $B(\zi,\e)\cap\OO$ the restriction $g|_Y$ is unbounded.
\end{Def}
\begin{prop}\label{prop2.7}
(\cite{2}). Let $\OO\subset\C^d$ be a domain and $g:\OO\ra\C$ a holomorphic function. If $g$ is totally unbounded, then $g$ is non-extendable.
\end{prop}
\begin{thm}\label{thm2.8}
Let $\OO\subset\C^d$ be a domain and $X=X(\OO)$ be a  topological vector space whose topology is induced by a complete metric, such that convergence in $X$ implies uniform convergence on compacta of $\OO$. Let $T$ be the identity or an operator of mixed partial derivation. We assume that for every ball  $B(\zi,\e)\subset\C^d$  with radius $\e>0$ centered at $\zi\in\partial\OO$ and for  every component $Y$ of $\OO\cap B(\zi,\e)$, there exists a function $h_Y$ in $X$ such that its restriction $Th_Y|_Y$ is unbounded. Then there exists a function $h\in X$ such that  $Th$ is totally unbounded on $\OO$. The set of $f\in X$ such that $Tf$ is totally unbounded on $\OO$ is a dense and $G_\de$ subset of $X$. Furthermore, the set of non-extendable functions in $X$ is a dense and $G_\de$ subset of $X$.
\end{thm}

\noindent
\begin{proof}
Let  $Z\subset\partial{\OO}$ be a countable set which is dense in $\partial{\OO}$. We write
$$
 \Ga=\{B(z,q):z\in Z, q\in \Q^+\}
 $$
for the collection of all balls with centers in $Z$ and rational radii, a denumerable set.  For  $B(z,q)\in \Ga$ let  $J_{B(z,q)}=\{Y \ \text{connected component of} \ \ B(z, q)\cap\OO\}$ and
$$
J=\bigcup_{B(z,q)\in \Ga}J _{B(z,q)}.
$$
For a subset $Y\subset\OO$ we write $S_Y=\{f\in X:Tf \, \text{is unbounded in} \,  Y\}.$

\smallskip
 \noindent
{\bf Claim}: $Tf$ is totally unbounded if and only if $f\in \bigcap_{Y\in J}S_Y$.\smallskip \\
{\em Proof of the claim}. It is clear that if $Tf$ is totally unbounded then $f\in \cap_{Y\in J}S_Y$.
To prove the converse suppose $B=B(x_0, \e_0)$ is a ball with $x_0\in\partial\OO$ and  $\e_0>0$,  and let $Y_0$ be a connected component of $B(x_0, \e_0)\cap\OO$. According to Lemma \ref{lem2.2} we have $\overline{Y_0}\cap\partial\OO\cap B\neq\emptyset$; let $x_1\in\overline{Y_0}\cap\partial\OO\cap B$, then since $B$ is open, there exists $\e_1>0$ such that $B(x_1, \e_1)\subset B$. Let $q\in \Q^+$ such that $0<q<\e_1$ then  $B(x_1,q)\subset B(x_1,\e_1)\subset B$. Since $x_1\in\partial\OO=\overline{Z}$ there exists $z_1\in Z$ such that $d(x_1,z_1)<\frac{q}{2}$, thus
$$
x_1\in B\Big(z_1,\frac{q}{2}\Big) \subset  B(x_1,q)\subset B.
$$
Setting  $\e_2=\frac{q}{2}-d(x_1,z_1)>0$ we have $B(x_1,\e_2)\subset B(z_1,\frac{q}{2})$. Since $x_1\in\overline{Y_0}$ there exists $x_2\in Y_0\cap B(x_1,\e_2)$. Thus
$$
x_2\in Y_0\cap B\Big(z_1,\frac{q}{2}\Big)\subset\OO\cap B\Big(z_1,\frac{q}{2}\Big).
$$
 Let $A$ be the connected component of $\OO\cap B(z_1,\frac{q}{2})$ containing $x_2$. Then $A\in J$, and since $A\subset\OO\cap B(z_1,\frac{q}{2})\subset\OO\cap B$ and $x_2\in A$, it follows that $A\subseteq Y_0$ as $A$ is connected and $Y_0$ is a connected component.

Now let  $f\in\cap_{Y\in J}S_Y$ then $f\in S_A$ and thus $Tf$ is unbounded on $A$. Since $A\subseteq Y_0$ it follows that   $Tf|_{Y_0}$ is unbounded. This finishes the proof of the claim.

We continue with the proof of Theorem.
If we prove that each $S_Y$, $Y\in J$, is $G_\de$ and dense in the complete metric space $X$, since $J$ is denumerable, then also $\cap_{Y\in J}S_Y$ is $G_\de$ and dense,  according to Baire's theorem. Thus it suffices to prove that each $S_Y$ is $G_\de$ and dense in $X$.

From \cite[Proposition 5.2]{7} we know that $S_Y$ is either empty or $G_\de$ and dense. By assumption there exists $h_Y\in S_Y$. Thus each $S_Y$ is $G_\de$ and dense in $X$.

We have proved that the set of $f\in X$, such that $Tf$ is totally unbounded is a $G_\de$ and dense subset of $X$. Proposition \ref{prop2.7} implies that there exists $h\in X$ such that $Th$ is non-extendable. Proposition \ref{prop2.4} implies that $h\in X$ is non extendable. Finally, Theorem \ref{thm2.5} implies that the set of non-extendable functions in $X$ is $G_\de$ and dense. This completes the proof. \end{proof}
\section{Localization of $H^{\infty}(\OO)$}\label{sec3}
\noindent

In this section we present a first application of Theorem \ref{thm2.8}.

Let $\OO\subset\C$ be a domain and $V\subset\OO$ an open set. We consider the set:
$$
X(\OO,V)=H(\OO)\cap H^\infty(V)=\{f\in H(\OO):f|_V\;\text{is bounded}\},
$$
see  \cite{4}. If $\overline{V}\subset\OO$ and $V$ is bounded, then obviously $X(\OO,V)=H(\OO)$ and the space is endowed with its usual Fr\'{e}ch\'{e}t topology. More generally, independently of the fact that $\overline{V}\subset\OO$ or $\overline{V}\cap\partial\OO\neq\emptyset$ the topology of $X(\OO,V)$ is the Fr\'{e}ch\'{e}t topology induced by the seminorms $\|f|_V\|_\infty$ and $\|f|_{K_m}\|_\infty$, $m=1,2,\ld$ where $\big\{K_m\big\}^\infty_{m=1}$ is an exhaustive sequence of compact subsets of $\OO$.

Obviously $X= X(\OO,V)$ satisfies the requirements of Theorem \ref{thm2.8}. Therefore, in order to prove that the set of functions with totally unbounded derivative is $G_\de$ and dense in this space it suffices to prove that for each ball $B(x_0,\e)$ centered at $x_0\in\partial\OO$ with radius $\e>0$ and $Y$ connected component of $B(x_0,\e)\cap\OO$ there exists an $h_Y\in X(\OO,V)$ such that $h'_{Y}$ is unbounded on $Y$. This requires some extra hypothesis on $\OO$.
\begin{thm}\label{thm3.1}
Let $\OO\subset\C$ be a domain and $V\subset\OO$ an open set such that
for all $ \al\in\overline{V}\cap\partial\OO $ there exists $ \bi\in\OO^c$ such that $[\al,\bi]\subseteq\OO^c$.
Then the set $\{f\in X(\OO,V):f'$ totally unbounded$\}$ is dense and $G_\de$ in $X(\OO,V)$; furthermore, the set of non-extendable functions in $X(\OO,V)$ is a $G_\de$ and dense subset.
\end{thm}
\begin{proof}
Consider a ball $B(x_0, \e)$ such that $x_0\in\partial\OO$ and $Y$ a connected component of $B(x_0, \e)\cap\OO$.
If $\overline{V}\cap\partial\OO\cap\overline{Y}\neq\emptyset$ then for $\al\in\overline{V}\cap\partial\OO\cap\overline{Y}\subset\overline{V}\cap\partial\OO$ there exists $\bi\in\OO^c$ such that $[\al,\bi]\subseteq\OO^c$. As a result a branch of logarithm can be chosen such that  $f(z)=\log\Big(\frac{z-\al}{z-\bi}\Big)$ is a holomorphic function on $\OO$ and $h_Y(z)=e^{-if(z)}$ has the desired property; that is, $h_Y(z)\in X(\OO,V)$ and $h'_{Y}$ is unbounded on $Y$. This follows easily from the fact that $|h_Y(z)|\in(e^{-\pi},e^\pi)$ for all $z\in\OO$ and $h'_{Y}(z)=(-i)h_Y(z)\cdot\frac{\al-\bi}{(z-\al)(z-\bi)}$ and $a\in\overline{Y}$.

If $\overline{V}\cap\partial\OO\cap\overline{Y}=\emptyset$, then according to Lemma \ref{lem2.2} there exists $\al\in\overline{Y}\cap\partial\OO$ which implies that $\al\notin\overline{V}$ and the function $h_Y=\frac{1}{z-\al}$ belongs to $X(\OO,V)$ and $h'_{Y}$ is unbounded on $Y$. The proof is completed as in Theorem \ref{thm2.8}.
\end{proof}
\section{Localization  of $A^p(\OO)$}\label{sec4}
\noindent

This section contains a second application of Theorem \ref{thm2.8}. We consider the case where $\OO$ is a Jordan domain in $\C$ and $J\subseteq\partial\OO$ a relatively open subset of its boundary. The space $A(\OO,J)$ consists of all functions $f\in H(\OO)$ such that $f$ can be continuously extended on $\OO\cup J$. The topology in this space is induced by the seminorms $\|f|_{\De_m}\|_\infty$ where
$$
\De_m=\bigg\{z\in\OO\cup J:\dist(z,\partial\OO\sm J)\ge\frac{1}{m}\bigg\}, \ \ m=1,2,\ld ,
$$
a sequence of compact subsets of $\OO\cup J$, see \cite{4}. We define the space $A^p(\OO,J)$ similarly to the space $A^0(\OO,J)=A(\OO,J)$.
Specifically a function $f$ belongs to $A^p(\OO,J)$ if $f$ is holomorphic on $\OO$ and for every $\el=0,1,2,\cdots,  p$,   the derivative of order $\el$ belongs to $A(\OO,J)$. The topology in this space is induced by the seminorms $\|f^{(\el)}|_{\De_m}\|_\infty$,  $0\le\el\le p$,  $m=1,2,\ld$,  where $\big\{\De_m\big\}^\infty_{m=1}$ are as above.

Obviously $A^p(\OO,J)$ satisfies the requirements of Theorem \ref{thm2.8}. Therefore, in order to prove that the set of functions with a totally unbounded derivative of order $p+1$ is $G_\de$ and dense in this space it suffices to prove that for each ball $B(x_0,\e)$ centered at $x_0\in\partial\OO$ with radius $\e>0$ and each connected component $Y$ of $B(x_0,\e)\cap\OO$ there exists a function $h_Y\in A^p(\OO,J)$ such that $h^{(p+1)}_Y$ is unbounded on $Y$.
\begin{thm}\label{thm4.1}
Let $\OO$ be a Jordan domain and $J\subset\partial\OO$ a relatively open subset of $\partial\OO$ such that for all $\al\in J$ there exists $\bi\in(\overline{\OO})^c$ such that $(\al,\bi]\subseteq(\overline{\OO})^c$. Then the set:
\[\{f\in A^p(\OO,J):f^{(p+1)} \ \ \text{is totally unbounded}\}
\]
is dense and $G_\de$ in $A^p(\OO,J)$.
Furthermore, the set of non-extendable functions in $A^p(\OO,J)$ is $G_\de$ and dense in $A^p(\OO,J)$.
\end{thm}
\begin{proof}
Consider a ball $B(x_0,\e)$ such that $x_0\in\partial\OO$, $\e>0$ and $Y$ a connected component of $B(x_0,\e)\cap\OO$.

If $J\cap\overline{Y}\neq\emptyset$, we pick $\al\in J\cap\overline{Y}$. By hypothesis there exists $\bi\in(\overline{\OO})^c$ such that $(\al,\bi]\subseteq(\overline{\OO})^c$. As a result the function $f(z)=\log\Big(\frac{z-\al}{z-\bi}\Big)$ is a holomorphic function on $\OO$ and its imaginary part $\Arg\Big(\frac{z-\al}{z-\bi}\Big)$ stays bounded on $\OO$.\\
\noindent
{\bf Claim}: The function $$
h_Y(z)=\frac{(z-\al)^{p+1}}{(p+1)!}\log\Big(\frac{z-\al}{z-\bi}\Big)
$$
has the desired property; that is, $h_Y\in A^p(\OO,J)$ and $h^{(p+1)}$ is unbounded on $Y$.\medskip\\
\noindent
{\em Proof of the claim}: A computation gives
\begin{align*}
h'_{Y}(z)&=\frac{(z-\al)^p}{p!}\log\bigg(\frac{z-\al}{z-\bi}\bigg)+\frac{(\al-\bi)}
{(p+1)!}\frac{(z-\al)^p}{z-\bi} \\
&=\frac{(z-\al)^p}{p!}\cdot\log\bigg(\frac{z-\al}{z-\bi}\bigg)
+\frac{(\al-\bi)}{(p+1)!}\left(q_1(z)+\frac{c_1}{z-\bi}\right).
\end{align*}
where $q_1(z)$ is a polynomial of degree $p-1$ and $c_1$ is a constant.
By induction,
\[
f^{(k)}(z)=\frac{(z-\al)^{p+1-k}}{(p+1-k)!}\cdot\log\bigg(\frac{z-\al}{z-\bi}\bigg)
+\sum^k_{n=1}\frac{(\al-\bi)}{(p+1-k+n)!}\cdot\frac{Q_{k-n+1}}{(z-\bi)^n},
\]
where
\[
\frac{Q_{k-n+1}}{(z-\bi)^n}=q^{n-1}_{k-n+1}(z)+(-1)^{n+1}\frac{(n-1)!c_{k-n+1}}
{(z-\bi)^n}
\]
where $q_{k-n+1}$ is the quotient of $(z-\al)^{p-k+n}$ divided by $z-\bi$ and $c_{k-n+1}$ is the remainder of the division; thus, for $k=p+1$ we have
\[
f^{(p+1)}(z)=\log\bigg(\frac{z-\al}{z-\bi}\bigg)+\sum^k_{n=1}\frac{(\al-\bi)}{n!}
\bigg[q^{n-1}_{p+2-n}(z)+(-1)^{n+1}\frac{(n-1)!}{(z-\bi)^n}c_{p+2-n}\bigg]
\]
which is the sum of $\log\Big(\frac{z-\al}{z-\bi}\Big)$ and a rational function bounded near $\al$. Therefore, $f^{(p+1)}$ is unbounded on $\overline{Y}\backepsilon\al$ and on $Y$ while $f^{(k)}\in A(\OO,J)$ since  it can be continuously extended on $J$ by setting $f^{(k)}(\al)=0$ for $0\le k\le p$, where we have used the fact that $\Im\log\frac{z-\al}{z-\bi}=\Arg\frac{z-\al}{z-\bi}$ is bounded on $\OO$. This completes the proof of the claim.

If $J\cap\overline{Y}=\emptyset$, then according to Lemma \ref{lem2.2}, there exists $\al\in\overline{Y}\cap\partial\OO$ which implies $\al\notin J$.
We consider the function $h_Y=\frac{1}{z-\al}$
which has the desired property. This completes the proof.
\end{proof}
\section{Localization of Bergman spaces}\label{sec5}
\noindent

A third application of Theorem \ref{thm2.8} is given now. Let $p\in(0,\infty)$, $\OO\subset\C$ a domain and $V\subset\OO$ a bounded open subset of $\OO$. We define $OL^p(\OO,V)$ to be the space of holomorphic functions on $\OO$ such that $\int_{V}|f|^pdxdy<\infty$, see \cite{2}, \cite{4}.

For $1\le p<\infty$ the topology of $OL^p(\OO,V)$ is the Fr\'{e}ch\'{e}t topology induced by the seminorms
\[
\|f\|_{1,p,V}=\bigg(\iint_V|f|^pdxdy\bigg)^{\frac{1}{p}} \ \ \text{and} \ \ \|f\|_{m,p,V}=\|f|_{K_m}\|_\infty
\]
for $m=2,3,\ld$ where $K_m$ is an exhaustive sequence of compact subsets of $\OO$.  The distance function is
\[
d_{p,V}(f,g)=\sum^{\infty}_{m=1}\frac{1}{2^m}\cdot\frac{\|f-g\|_{m,p,V}}
{1+\|f-g\|_{m,p,V}}.
\]
For $0\le p<1$ we set
\[
d_{1,p,V}(f,g)=\iint_V|f-g|^pdxdy
\]
and
\[
d_{m,p,V}(f,g)=\sup_{z\in K_m}|f(z)-g(z)| \ \ m=2,3,\ld
\]
and the distance function in $OL^p(\OO,V)$ is given by
\[
d_{p,V}(f,g)=\sum^{\infty}_{m=1}\frac{1}{2^m}\cdot\frac{d_{m,p,V}(f,g)}
{1+d_{m,p,V}(f,g)}.
\]
In both cases, that is fro $0<p<\infty$, $OL^p(\OO,V)$ satisfies the requirements of Theorem \ref{thm2.8}.

In addition we consider  the spaces $X_q$ for  $0<q\leq \infty$ which are defined by
$$
X_q=\bigcap_{p<q}OL^p(\OO,V).
$$
 Convergence in $X_q$  is taken to coincide with  convergence in $OL^p(\OO,V)$ for all $p<q$. It is easy to see that  if $p_n$, $n=1,2,3,\ld$ is a strictly increasing sequence converging to $q$, then convergence in $X_q$ is equivalent to convergence in $OL^{p_n}(\OO,V)$ for all $n=1,2,\ld\;.$, i.e. it is induced by  the distance function
\[
d_{X_q,V}(f,g)=\sum^{\infty}_{n=1}\frac{1}{2^n}\cdot\frac{d_{p_n,V}(f,g)}
{1+d_{p_n,V}(f,g)}.
\]

One can check that the space $X_q$ satisfies the requirements of Theorem \ref{thm2.8}.
Thus, in order to prove that in each of the previously mentioned spaces the set of totally unbounded functions and the set of non-extendable functions are $G_\de$ and dense it suffices to find the function $h_Y$ of Theorem \ref{thm2.8}.
\begin{thm}\label{thm5.1}
Let $\OO\subset\C$ be a domain and $V\subset\OO$ a bounded open set, such that for each $\al\in\overline{V}\cap\partial\OO$ there exists $\bi\in\OO^c$ with $[\al,\bi]\subseteq\OO^c$. Let $X$ be any of the spaces $OL^p(\OO,V)$ or $X_q$
as above.  Then the set of totally unbounded functions in $X$ is a $G_\de$ and dense subset of $X$. Furthermore, the set of non-extendable functions in $X$ is also a $G_\de$ and dense subset of $X$.
\end{thm}
\begin{proof}
Let $B(\zi,\e)$ be a ball centered at $\zi\in\partial\OO$ with radius $\e>0$ and let $Y$ be a connected component of $\OO\cap B(\zi,\e)$. According to Theorem \ref{thm2.8} it suffices to find a function $h_Y\in X$ whose restriction on $Y$ is unbounded.

If $\overline{Y}\cap\overline{V}\cap\partial\OO\neq\emptyset$ we take  $\al\in\overline{Y}\cap\overline{V}\cap\partial\OO$ and by assumption there exists $\bi\in\OO^c$ such that $[\al,\bi]\subseteq\OO^c$. Hence the function $f(z)=\log\Big(\frac{z-\al}{a-\bi}\Big)$ is holomorphic on $\OO$. We set $h_Y(z)=\log\Big(\frac{z-\al}{z-\bi}\Big)$. This function belongs to all spaces $X$ mentioned above. This can be easily seen using the following facts:
\begin{enumerate}
\item[1)] the triangle  inequality of the $p$-norm
\item[2)] the fact that the $\arg\Big(\frac{z-\al}{z-\bi}\Big)$ is bounded on $V$
\item[3)] the fact that the set $V$ is bounded which implies $(\log|z-\al|)|z-\al|^{\frac{1}{p}}\le C_{p,R}<\infty$ for all $z$ belonging to a ball $B(\al,R)$ containing $V$
\item[4)] the fact that $\int_{B(\al,R)}|\log|z-\al|\,|^p dxdy<\infty$ for each $p\in(0,\infty)$ which follows easily using polar coordinates with center $\al$.
\end{enumerate}

If $\overline{Y}\cap\overline{V}\cap\partial\OO=\emptyset$, then according to Lemma \ref{lem2.2} we have $\overline{Y}\cap\partial\OO\neq\emptyset$. In this case we take  $\al\in\overline{Y}\cap\partial\OO$ which implies $\al\notin\overline{V}$ and the function $h_Y=\frac{1}{z-\al}$ is unbounded on $Y$ and belongs to $X$, because it is bounded on $V$.
This completes the proof according to Theorem \ref{thm2.8}.
\end{proof}
\section{Localization of Hardy spaces in the disc}\label{sec6}
\noindent

We now consider localized  Hardy spaces in the unit disc $D$ in $\C$.
Let $0<p<\infty$ and $J=\{e^{i\thi}:\al\le\thi\le\bi\}$, $\al<\bi\le\al+2\pi$,  a closed arc of the unit circle $\T$.  We define $H^p(D,J)$ to be the space of holomorphic functions on $D$ such that
$$
\sup_{0\le r<1}\int^\bi_\al\big|f(re^{i\thi})\big|^pd\thi<\infty.
$$
 For $1\le p<\infty$ the topology of $H^p(D,J)$ is the Fr\'{e}ch\'{e}t topology induced by the seminorms
$$
\|f\|_{1,p,J}=\sup_{0\le r<1}\bigg(\int^\bi_\al|f(re^{i\thi})|^pd\thi\bigg)^{1/p}
$$
and
$$ \ \
\|f\|_{m,p,J}=\sup_{|z|\le1-\frac{1}{m}}|f(z)| \ \ m=2,3,\ld\;,
$$
 and the distance function is
\[
d_{p,J}(f,g)=\sum^\infty_{m=1}\frac{1}{2^m}\frac{\|f-g\|_{m,p,J}}
{1+\|f-g\|_{m,p,J}}.
\]
For $0<p<1$ we set
\[
d_{1,J,p}(f,g)=\sup_{0\le r<1}\int^\bi_\al|f(re^{i\thi})-g(re^{i\thi})|^pd\thi,
\]
\[
d_{m,p,J}(f,g)=\sup_{|z|\le1-\frac{1}{m}}|f(z)-g(z)|, \ \ m=2,3,\ld
\]
so the distance function  is given by
\[
d_{p,J}(f,g)=\sum^\infty_{m=1}\frac{1}{2^m}\frac{d_{m,p,J}(f,g)}
{1+d_{m,p,J}(f,g)}.
\]
It is easy to check that for $0<p<\infty$ the space $H^p(D,J)$ satisfies the requirements of the Theorem \ref{thm2.8}.

We also consider the spaces $X_q(D,J)$ for  $q\in(0,\infty]$ defined as  $X_q(D,J)=\cap_{p<q}H^p(D,J)$. Convergence in this space is equivalent to convergence in $H^p(D,J)$ for all $p<q$, which coincides with convergence in
$H^{p_n}(D,J)$ for all $n$, where $p_n$, $n=1,2,\ld$ is a strictly increasing sequence convergent to $q$.  The distance function in $X_q(D,J)$ is given by
\[
d_{X_{q,J}}(f, g)=\sum^\infty_{n=1}\frac{1}{2^n}\frac{d_{p_n,J}(f, g)}{1+d_{p_n,J}(f, g)}.
\]
One can check that the space $X_q(D,J)$ satisfies the requirements of Theorem \ref{thm2.8}. If we consider a denumerable family of spaces of the previous type with varying $J$, $p$ or $q$ then the intersection of these spaces endowed with a distance function as above also satisfies the requirements of Theorem \ref{thm2.8}.

In addition  if $O\subset\T$ is a relatively open set and $J_m$ are defined,
\[
J_m=\bigg\{z\in O:\dist(z,\T-O)\ge\frac{1}{m}\bigg\}, \ \ m=1,2,\ld\;,
\]
then $J_m$ are  an exhaustive family of compact subsets of $O$, and we define the space $H^p(D,O)$ by
\[
H^p(D,O)=\Big\{f\in H(D):\sup_{0\le r<1}\int_{J_m}|f(re^{i\thi})|^pd\thi<\infty \;\text{for all}\;m=1,2,\ld\Big\}.
\]
We notice that each $J_m$ is a finite disjoint union of closed subarcs of $\T$. Thus $H^p(D,O)$ coincides with the set of $f$ holomorphic on $D$ such that $f\in H^p(D,J)$ for all closed arcs $J\subset O$ and endowed with the  natural topology induced by $H^p(D,J_m)$, $m=1,2,\ld$, and  satisfies the requirements of Theorem \ref{thm2.8}. Furthermore, spaces that are  denumerable intersections of such spaces, satisfy the requirements of Theorem \ref{thm2.8}.

Thus, in order to prove that in each one of the  previous spaces the set of totally unbounded functions and the set of non extendable functions are $G_\de$ and dense it suffices to find the function $h_Y$ of Theorem \ref{thm2.8}.

Let $Y=D\cap\{z:|z-1|<\e\}$, $\e>0$. It is well known that the function  $h_Y(z)=\log(1-z)$ belongs to all Hardy spaces  $H^p(D)$, $0<p<\infty$, therefore it also belongs to all localized Hardy spaces or their intersections as  considered above. This function $h_Y$ is unbounded on $Y$, and its rotations $h_{Y, \f} (z)=\log(1-e^{-i\f}z)$ are unbounded $Y_{\f}=D\cap\{z:|z-e^{i\f}|<\e\}$.  Thus, Theorem \ref{thm2.8} implies the following.
\begin{thm}\label{thm6.1}
Let $X$ be any localized $H^p(D,J)$, $J\subset\T$ closed arc, $0<p<\infty$ or any denumerable intersection of them. Then the set of totally unbounded functions in $X$ is a $G_\de$ and dense subset of $X$. Furthermore, the set of non-extendable functions in $X$ is also a $G_\de$ and dense subset of $X$.
\end{thm}


It is interesting to observe that, as in the case of classical Hardy spaces, functions in a localized Hardy space posses nontangential limits on the appropriate arc. 
\begin{prop}\label{prop6.2}
Let $J=\{e^{i\thi}:A\le\thi\le B\}\subset\T$, $A<B\le A+2\pi$, be a closed arc and $0<p<\infty$. If $f\in H^p(D,J)$ then  for almost all $e^{i\thi}$ in $J$ with respect to arc length measure the non-tangential limit of $f(z)$ at $e^{i\thi}$ exists.
\end{prop}

\begin{proof} Since $\sup_{0\le r<1}\int^B_A|f(re^{i\thi})|^pd\thi=M<\infty $ we have 
$$
\int^1_{0}\int^B_{A}|f(re^{i\thi})|^pd\thi dr<\infty.
$$
By Fubini's theorem for almost all $\al$ and $\bi$, $A<\al<\bi<B$ we have
$$
\int^1_{0}|f(re^{i\al})|^pdr=M_\al<\infty \ \ \text{and} \ \ \int^1_{0}|f(re^{i\bi})|^pdr=M_\bi<\infty
$$
For such $\al$ and $\bi$ it suffices to prove that for almost all $\thi$ in $(\al,\bi)$ the non-tangential limit of $f(z)$ at $e^{i\thi}$ exists.

We consider the sector $\OO=\{re^{i\thi}:0<r<1,\;\thi\in(\al,\bi)\}$ whose boundary has finite length. According to the discussion in \cite{1} it suffices to show that $f|_\OO$ belongs to the Hardy space $E^p(\OO)$, consisting of function on $\OO$ whose integral of $p$-powers along a sequence of rectifiable Jordan curves tending to the boundary remains bounded. 

Let $\e>0$ and $r_\e\in (1-\frac{\e}{2},1)$. Then
$$
\int^\bi_\al|f(r_\e e^{i\thi})|^pd\thi\le\sup_{0\le r<1}\int^\bi_\al|f(re^{i\thi})|^pd\thi=M<\infty
$$
The function $f$ is uniformly continuous on the compact set $\{re^{i\thi}:0\le r\le r_\e,a\le\thi\le\bi\}\subset D$. We can find segments $[P,Q],[P,R]$ with $P,Q,R\in\OO$ and  $|Q|=|R|=r_\e$, such that  $[P,Q]$ is parallel to $[0,e^{i\al}]$, $[P,R]$ is parallel to $[0,e^{i\bi}]$ and such that the  curve $[P,Q]\cup\overset{\frown }{QR}\cup[R,P]$ is $\e$-close to $\partial\OO$, where $\stackrel\frown{QR}$ is an arc in the circle with center 0 and radius $r_\e$, in such a way that  $\int_{[P,Q]}|f(z)|^p|dz|<M_\al+1$ and $\int_{[P,R]}|f(z)|^p|dz|<M_\bi+1$. It follows that
$$
\int_{[P,Q]\cup\stackrel\frown {QR}\cup[R,P]}|f(z)|^p|dz|<M+M_\al+M_\bi+2.
$$
Since $M+M_\al+M_\bi+2$ is independent of $\e$, it follows that $f|_\OO$ belongs to the Hardy space $E^p(\OO)$ and the proof is complete.
\end{proof}


\section{Functions with Taylor coefficients in $l^p$, $c_o$ and $l^{\infty}$.}\label{sec7}

We next consider the spaces $l^p_a$, $c_{0,a}$ and $l^{\infty}_a$  consisting of holomorphic functions in the unit disc whose sequence of Taylor coefficients is $p$-summable, a null sequence, or a bounded sequence respectively. More precisely for  $0<p\leq \infty$ let
$$
l^p_a=\{ f(z)=\sum_{n=0}^{\infty}a_nz^n:  (a_n)\in l^p \}
$$
where $l^p$ is the classical space of complex sequences $(a_n)$ such that $\n{(a_n)}_p=(\sum_{n=0}^{\infty}|a_n|^p)^{1/p}<\infty$, and $\n{(a_n)}_{\infty}=\sup_n|a_n|$ for $p=\infty$. For  $c_{0,a}$ the sequence $(a_n)$ tends to $0$.  It is clear that each power series in these spaces  has radius of convergence at least $1$ since in all cases $(a_n)$ is a bounded sequence,  so the series  defines an analytic function $f$ on the unit disc $D$.
For $1\leq p\leq \infty$ the spaces $l^p_a$ are Banach spaces with the norm inherited from  $l^p$, and the same is true for $c_{0,a}$ with  norm inherited from $l_a^{\infty}$.  For $0<p<1$ the distance function
$$
d_p(f, g)= \sum^\infty_{n=0}|f(n)-g(n)|^p.
$$
(where we use the notation $f(n)$ for the $n$th Taylor coefficient of an $f$),  
is a complete metric and makes $l^p_a$ into a complete linear topological   space.

If $X$ is any of the above spaces then it can be checked easily that convergence in $X$ of  a sequence $(f_n)$, $f_n(z)=\sum^\infty_{k=0}f_n(k) z^k$,  to a function $f(z)=\sum^\infty_{k=0}f(k) z^k$ implies
$$
\sup_{k}|f_n(k)-f(k)|\to 0, \,\,\text{as}\,\, n\to \infty.
$$
Thus if $0<r<1$ then for $|z|<r$,
\begin{align*}
|f_n(z)-f(z)|&\leq 
\sum^\infty_{k=0}|f_n(k)-f(k)||z|^k
 \leq \sup_{k}|f_n(k)-f(k)|\sum_{k=0}^{\infty}r^k\\
  &=\frac{1}{1-r} \sup_{k}|f_n(k)-f(k)|, 
\end{align*}
and it follows that convergence in norm  implies  uniform convergence on compact subsets  of $D$.   Thus $X$ satisfies  the requirements of Theorem \ref{thm2.8}.

We may also consider the spaces  $X_q=\cap_{p>q}\el^p_a$ for  $0\leq q<\infty$ with the distance function
$d_q(f, g)=\sum^\infty_{n=0}\frac{1}{2^n}\frac{d_{p_n}}{1+d_{p_n}}
$
where $(p_n)$ is a strictly decreasing sequence of indices converging to $q$, and $d_{p_n}$ is the distance function in $l^{p_n}_a$. Then it is easily verified that the space $X_q$, $0\le q<\infty$, also  satisfies  the requirements of Theorem \ref{thm2.8}.

\begin{thm}\label{thm7.1}
Let $X$ be any  of the spaces $l^p_a$, $(1<p\leq \infty)$, $c_{0, a}$  or $X_q$, $(1\le q<\infty)$. Then the set of totally unbounded functions in $X$ is a $G_\de$ and dense subset of $X$, and the same is true  for the set of non-extendable functions in $X$.
\end{thm}
\begin{proof}
To apply Theorem \ref{thm2.8} it suffices to find the function $h_Y$ in $X$. Let $Y=D\cap\{z:|z-1|<\e\}$, $\e>0$. Then function  $h_Y(z)=\sum^\infty_{k=1}\frac{1}{k}z^k$  satisfies $\lim_{r\to 1^-}h_Y(r)=\infty$ and $h_Y$ is unbounded on $Y$. It  belongs to the space $X$ because $\sum^\infty_{k=1}\frac{1}{k^p}<\infty$ for all $1<p<\infty$. If $Y=D\cap\{z:|z-e^{i\thi}|<\e\}$ then the rotated function $h_Y(z)=\sum^\infty_{k=1}\frac{1}{k}e^{-ik\thi}z^k$ has the same properties on $Y$, and Theorem \ref{thm2.8} applies and gives the assertion.
\end{proof}

If $f\in l^p_a$ with $0<p\leq 1$ then $f$ extends continuously on the closed disc $\overline{D}$ and is bounded, but the derivative  $f'$ can be totally unbounded.
\begin{thm}\label{thm7.2}
Let $X$ be any  of the space $l^p_a$, $(0<p\leq \infty)$, $c_{0,a}$  or $X_q$, $0\le q<\infty$. Then the set of functions $f\in X$ such that the derivative $f'$ is totally unbounded is a $G_\de$ and dense subset of $X$. The same holds for the set of non-extendable functions in $X$.
\end{thm}
\begin{proof}
We will apply again  Theorem \ref{thm2.8}. Let $Y=D\cap\{z:|z-1|<\e\}$.  
The function $h_Y(z)=\sum^\infty_{n=0}\frac{1}{2^n}z^{2^n}$
 belongs to $X$ since $\sum^\infty_{n=0}\left(\frac{1}{2^n}\right)^p=\frac{2^p}{2^p-1}<\infty$
for all $p>0$. Its derivative  $h'_Y(z)=\sum^\infty_{n=0}z^{2^n-1}$ satisfies $\lim_{r\to 1^-}h'_Y(r)=\infty$ and $h'_Y$ is unbounded on $Y$. If $Y=D\cap\{z:|z-e^{i\thi}|<\e\}$ then  $h_Y(z)=\sum^\infty_{n=0}\frac{e^{-i2^n\thi}}{2^n} z^{2^n}$ has the same property on $Y$. Thus Theorem \ref{thm2.8} applies.
\end{proof}

\textbf{Remark}.  Since the sets $D\cap\{z:|z-e^{i\thi}|<\e\}$ are convex, it follows easily that, if $g$ is holomorphic on $D$ and  the derivative $g^{(k)}$ for some $k\ge 0$ is totally unbounded on $D$ then for every $\el\ge k$ the derivative $g^{(\el)}$ is also totally unbounded. Thus the conclusion of  Theorem \ref{thm7.1} implies that the set of $f\in X$ such that every derivative $f^{(\el)}$, $\el\ge0$, is totally unbounded is a $G_\de$ and dense subset of $X$. Similarly in Theorem \ref{thm7.2} the set of $f\in X$ such that every derivative $f^{(\el)}$, $\el\ge1$, is totally unbounded is a $G_\de$ and dense subset of $X$.

Finally, we mention that in the spaces $X$ considered in this section a necessary and sufficient condition so that the set of totally unbounded functions $f$ in $X$ is $G_\de$ and dense is that there exists an unbounded function $g$ in $X$. One can easily prove the above statement not only for the functions themselves but for their derivations, as well, but we do not need this fact. We mention it because it furnishes  interesting information.

\end{document}